\documentclass[preprint,12pt]{elsarticle}
\usepackage[T1]{fontenc}
\usepackage[utf8]{inputenc}
\usepackage[english]{babel}
\usepackage{lmodern,microtype}
\usepackage{amsmath,amssymb,amsthm,mathtools}
\usepackage{needspace,etoolbox}
\usepackage{xurl,hyperref,bookmark}
\hypersetup{unicode=true,hidelinks,
 pdftitle={A proof of Kahane's conjecture in the Wiener algebra},
 pdfauthor={Przemys{\l}aw Ohrysko},
 pdfsubject={Harmonic analysis and the Wiener algebra},
 pdfkeywords={Wiener algebra, Kahane conjecture, Beurling--Helson theorem, norms of powers, non-harmonic Fourier series},
 bookmarksnumbered=true}
\allowdisplaybreaks[1]
\numberwithin{equation}{section}
\newtheorem{theorem}{Theorem}[section]
\newtheorem{lemma}[theorem]{Lemma}
\newtheorem{proposition}[theorem]{Proposition}
\newtheorem{corollary}[theorem]{Corollary}
\theoremstyle{remark}
\newtheorem{remark}[theorem]{Remark}
\newtheorem{example}[theorem]{Example}
\newcommand{\T}{\mathbb T}
\newcommand{\R}{\mathbb R}
\newcommand{\C}{\mathbb C}
\newcommand{\Z}{\mathbb Z}
\newcommand{\N}{\mathbb N}
\newcommand{\dd}{\,\mathrm d}
\newcommand{\norm}[1]{\left\lVert#1\right\rVert}

\BeforeBeginEnvironment{theorem}{\Needspace{5\baselineskip}}
\BeforeBeginEnvironment{lemma}{\Needspace{5\baselineskip}}
\BeforeBeginEnvironment{corollary}{\Needspace{5\baselineskip}}
\BeforeBeginEnvironment{proposition}{\Needspace{5\baselineskip}}
\pretocmd{\section}{\Needspace{8\baselineskip}}{}{}
\pretocmd{\subsection}{\Needspace{8\baselineskip}}{}{}
\makeatletter
\def\ps@pprintTitle{%
  \let\@oddhead\@empty
  \let\@evenhead\@empty
  \def\@oddfoot{\footnotesize\itshape Preprint\hfill October 10, 2026}%
  \let\@evenfoot\@oddfoot}
\makeatother

\begin{document}
\begin{frontmatter}
\title{A proof of Kahane's conjecture in the Wiener algebra}
\author[uw]{Przemys{\l}aw Ohrysko\corref{cor1}}
\cortext[cor1]{Corresponding author.}
\ead{P.Ohrysko@mimuw.edu.pl}
\address[uw]{University of Warsaw, Faculty of Mathematics, Informatics and Mechanics, Institute of Mathematics, Banacha 2, 02-097 Warsaw, Poland}

\begin{abstract}
We prove Kahane's conjecture: if a continuous unimodular function
$F$ on the circle satisfies
$\norm{F^n}_{A(\T)}=o(\log|n|)$ as $|n|\to\infty$, $n\in\Z$, then
$F(x)=e^{i(mx+b)}$ for some $m\in\Z$ and $b\in\R$.
The main step is a logarithmic lower bound for averages:
for every nonconstant real-valued $\varphi\in A(\T)$,
\[
 \frac1T\int_a^{a+T}\norm{e^{it\varphi}}_{A(\T)}\dd t
 \ge c_\varphi\log T
 \qquad(a\in\R,\ T\ge T_\varphi),
\]
where $c_\varphi>0$ and $T_\varphi>1$ are independent of $a$.
The proof uses exact sampling on finite subgroups of the circle,
uniformly bounded localized modulations, and the weighted non-harmonic
Littlewood inequality of Jaming--Kellay--Saba. A suitable cutoff and
a small linear perturbation of the phase ensure that a fixed
positive proportion of the coefficients at the beginning of the
frequency ordering have modulus one. Passing to integer powers and removing the degree
of the circle map yield Kahane's conjecture without any additional
smoothness assumptions. We also give equivalent phase, convolution,
and operator formulations, and derive consequences for homomorphisms
of weighted Wiener algebras and for higher-dimensional tori.
\end{abstract}

\end{frontmatter}

\section{Historical introduction and main results}\label{sec:intro}

In 1953, Beurling and Helson~\cite{BH} proved that uniform boundedness
of the norms of powers in the Wiener algebra forces a circle map
to have a very rigid form. To recall their theorem, identify
$\T$ with $\R/(2\pi\Z)$ and let $A(\T)$ denote the algebra of
functions with absolutely convergent Fourier series, equipped with
the norm
\[
 \norm f_A=\sum_{k\in\Z}|\widehat f(k)|,
 \qquad
 \widehat f(k)=\frac1{2\pi}\int_0^{2\pi}f(x)e^{-ikx}\dd x.
\]
In its circle formulation, the Beurling--Helson theorem states that
if $F\in C(\T)$, $|F|=1$, and
\begin{equation}\label{eq:BH}
 \sup_{n\in\Z}\norm{F^n}_A<\infty,
\end{equation}
then $F(x)=e^{i(mx+b)}$ for some $m\in\Z$, $b\in\R$.
In other words, uniform boundedness of the norms of powers
characterizes affine maps $\T\to\T$.

Kahane proposed replacing boundedness in~\eqref{eq:BH} by
the condition
\begin{equation}\label{eq:kahane-condition}
 \norm{F^n}_A=o(\log|n|)\qquad(|n|\to\infty,\ n\in\Z).
\end{equation}
This question appears in his address to the International Congress
of Mathematicians in Stockholm in 1962, published in
1963~\cite{KahaneICM,LebedevBH}; see also his monograph~\cite[p.~87]{Kahane}.
The logarithmic scale is natural, and the condition cannot be replaced
by $O(\log|n|)$: for continuous piecewise affine circle maps that are
not globally affine, the norms of powers have order $\log|n|$;
see~\cite[Ch.~VI, \S2, pp.~75--76]{Kahane} and~\cite{LebedevBH}.
We give a concrete example in Section~\ref{sec:remarks}.

It is useful to distinguish two directions in the study of the
quantitative problem. The first concerns estimates that depend on
regularity of the phase. In~\cite{LebedevSmooth}, Lebedev studied
the growth of $\norm{e^{it\varphi}}_{A_p}$, $1\le p<2$, for real-valued
$C^1$ phases, where $\norm f_{A_p}=\norm{\widehat f}_{\ell^p}$.
The second concerns Kahane's conjecture under the sole assumption
that the circle map is continuous. In this direction,
Lebedev~\cite{LebedevBH} obtained the first general quantitative
strengthening of the classical theorem in 2012: it suffices to assume
\begin{equation}\label{eq:lebedev-bound}
 \norm{F^n}_A
 =o\!\left(\left(\frac{\log\log|n|}{\log\log\log|n|}\right)^{1/12}\right).
\end{equation}
His proof combines Diophantine approximation with analysis on finite
cyclic groups and uses the result of Green and Konyagin~\cite{GK}
on the Littlewood problem modulo a prime.

Konyagin and Shkredov~\cite{KS} refined this argument.
Their 2015 theorem gives the same conclusion under the assumption
\begin{equation}\label{eq:KS-bound}
 \norm{F^n}_A
 =o\!\left(\frac{(\log|n|)^{1/22}}{(\log\log|n|)^{3/11}}\right).
\end{equation}
Further progress was made by Sanders~\cite[Theorem~1.4]{Sanders} in a preprint
made available in April 2026. Sanders proved that it is enough to have
\begin{equation}\label{eq:sanders-bound}
 \norm{F^n}_A=O\!\left((\log|n|)^{1/8-\varepsilon}\right)
 \quad\text{for some }0<\varepsilon<\tfrac18.
\end{equation}
In~\eqref{eq:lebedev-bound}--\eqref{eq:sanders-bound}, the limit is
$|n|\to\infty$, $n\in\Z$.
All these results apply to continuous maps without additional
differentiability assumptions. These sufficient conditions still fall short of the threshold $o(\log|n|)$
proposed by Kahane.

The purpose of this paper is to reach that threshold.
For a continuous function outside $A(\T)$, we set
$\norm f_A=+\infty$; in particular, condition~\eqref{eq:kahane-condition}
requires that $F^n$ belong to $A(\T)$ for all sufficiently large $|n|$.

\begin{theorem}[Kahane's conjecture]\label{thm:kahane}
Let $F\in C(\T)$ satisfy $|F(x)|=1$ for every $x\in\T$.
If
\[
 \norm{F^n}_A=o(\log|n|)\qquad(|n|\to\infty,\ n\in\Z),
\]
then there exist $m\in\Z$ and $b\in\R$ such that
$F(x)=e^{i(mx+b)}$ for all $x\in\T$.
\end{theorem}

Under the assumption $|F|=1$, the two-sided condition is equivalent
to $\norm{F^n}_A=o(\log n)$ as $n\to+\infty$, $n\in\N$.
Indeed, complex conjugation preserves the $A(\T)$ norm, since
$\widehat{\overline g}(k)=\overline{\widehat g(-k)}$, and hence
\[
 F^{-n}=\overline{F^n},\qquad
 \norm{F^{-n}}_A=\norm{F^n}_A\qquad(n\in\Z).
\]
Thus it is enough to consider positive powers in the proof,
as in Konyagin and Shkredov~\cite[Section~4]{KS}.
We retain the standard formulation with $n\in\Z$ and $\log|n|$
in the statement.

Our proof does not establish a pointwise lower bound for each
individual norm $\norm{F^n}_A$. Instead, we prove a lower bound
for averages, uniform in the position of the averaging interval.
For a real-valued $\varphi\in A(\T)$, write
\begin{equation}\label{eq:M-Q}
 M_\varphi(t)=\norm{e^{it\varphi}}_A,
 \qquad
 \mathcal Q_\varphi(a,T)=\frac1T\int_a^{a+T}M_\varphi(t)\dd t.
\end{equation}
The exponential series in the Banach algebra ensures that
$M_\varphi$ is finite and continuous on $\R$.

\begin{theorem}[Logarithmic lower bound for averages]\label{thm:mean}
For every nonconstant real-valued $\varphi\in A(\T)$, there exist
$c_\varphi>0$ and $T_\varphi>1$ such that
\begin{equation}\label{eq:mean-statement}
 \mathcal Q_\varphi(a,T)\ge c_\varphi\log T
 \qquad(a\in\R,\ T\ge T_\varphi).
\end{equation}
\end{theorem}

No smoothness assumption on the phase is needed in this theorem.
It also yields a uniform estimate for averages over consecutive
integer powers of a unimodular function that is not a scalar multiple
of a character (Corollary~\ref{cor:unimodular-blocks}).

The analytic tool underlying the proof comes from the Littlewood
problem. Its classical version for exponential sums with integer
frequencies was settled independently by Konyagin~\cite{Konyagin81}
and by McGehee, Pigno, and Smith~\cite{MPS}.
Nazarov~\cite{Nazarov} obtained a non-harmonic version on a finite
interval, with a lower bound depending on the separation of the
frequencies but not on their number.
We use an explicit version of such an inequality, proved by
Jaming, Kellay, and Saba~\cite{JKS}.
The relevant quantity is a weighted sum of the moduli of the
coefficients in \emph{increasing frequency order}, rather than
merely the number of large coefficients.

The mechanism of our argument is as follows. We localize the phase
using a function $\chi$ that equals one where the phase takes small
values and vanishes near the endpoints of the chosen coordinate
interval on the circle. For a grid with $N$ points, we choose a single
small parameter $s_N$ so that the values $\varphi(x_j)-s_Nx_j$
are separated by at least $1/(16N^3)$. Exact sampling then turns
the function $\chi(x)e^{it(\varphi(x)-s_Nx)}$ into a family of
exponential sums in the variable $t$. In each of these sums, an
initial positive proportion of the coefficients have modulus one.
The non-harmonic inequality therefore gives a lower bound of order
$\log N$ on every interval of length at least $2304\pi N^3$.
Since the norm bound for the localized modulation is independent of
$ts_N$, we may choose $N$ of order $T^{1/3}$ and
obtain~\eqref{eq:mean-statement}.

\medskip
\noindent\textbf{Organization of the paper.}
Section~\ref{sec:algebra} recalls the required properties of the
Wiener algebra, including the existence of a logarithm after removing
the degree, and proves the equivalence of the phase, convolution,
and operator formulations of Kahane's conjecture.
Section~\ref{sec:tools} states the precise non-harmonic inequality
and proves the sampling and modulation lemmas.
Section~\ref{sec:geometry} constructs the cutoff and separates the
sampled phase values, and Section~\ref{sec:meanproof} proves
Theorem~\ref{thm:mean}. The passage to integer powers and the proof
of Theorem~\ref{thm:kahane} are given in
Section~\ref{sec:kahaneproof}.
Section~\ref{sec:consequences} treats composition operators,
weighted Wiener algebras, and higher-dimensional tori.
We conclude in Section~\ref{sec:remarks} with sharpness and a
comparison with related questions not settled by these arguments.
The core proof occupies Sections~\ref{sec:tools}--\ref{sec:kahaneproof};
Subsection~\ref{subsec:operator-equivalences} may be skipped on a first reading.

\section{The Wiener algebra and equivalent formulations}\label{sec:algebra}

\subsection{Basic properties and phase logarithms}

Throughout the paper, Haar measure on $\T$ is normalized to have
mass one. The $L^2(\T)$ norms refer to this measure, whereas
integration over intervals of the real line is with respect to
ordinary Lebesgue measure. All logarithms are natural.
We write $C(\T;\R)$ and $A(\T;\R)$ for the real-valued
elements of $C(\T)$ and $A(\T)$, respectively; the same
convention is used on higher-dimensional tori.

The algebra $A(\T)$ is a commutative unital Banach algebra,
$\norm1_A=1$, and, for $f,g\in A(\T)$,
\begin{equation}\label{eq:algebra}
 \norm{fg}_A\le\norm f_A\norm g_A,
 \qquad \norm f_\infty\le\norm f_A,
 \qquad \norm{\overline f}_A=\norm f_A.
\end{equation}
Translations of the argument and multiplication by $e^{imx}$,
$m\in\Z$, are isometries of $A(\T)$. All these properties follow
directly from the corresponding identities for Fourier coefficients.
For a real-valued $\varphi\in A(\T)$, the series defining
$e^{it\varphi}$ converges in the algebra norm, uniformly for $t$
in compact intervals. In particular,
\begin{equation}\label{eq:M-properties}
 M_\varphi(-t)=M_\varphi(t),\qquad M_\varphi(t)\ge1,\qquad
 M_\varphi(t+u)\le M_\varphi(t)M_\varphi(u).
\end{equation}

We shall also use the elementary embedding
$W^{1,2}(\T)\subset A(\T)$\footnote{Here $W^{1,2}(\T)$ denotes the
Sobolev space of functions $h\in L^2(\T)$ whose first weak derivative
$h'$ (taken on the circle, in the sense of periodic distributions)
also belongs to $L^2(\T)$.}:
\begin{equation}\label{eq:sobolev}
 \norm h_A\le C_S\bigl(\norm h_2^2+\norm{h'}_2^2\bigr)^{1/2},
 \qquad
 C_S=\left(\sum_{k\in\Z}\frac1{1+k^2}\right)^{1/2}.
\end{equation}
Indeed, apply the Cauchy--Schwarz inequality to
$\sum_k(1+k^2)^{-1/2}(1+k^2)^{1/2}|\widehat h(k)|$
and use Parseval's identity.

We include proofs of the following standard facts about logarithms
in $A(\T)$ so that the passage from a real-valued phase to an arbitrary
circle map requires no additional assumptions.

\begin{lemma}\label{lem:spectrum}
For $h\in A(\T)$, the spectrum and spectral radius satisfy
\[
 \sigma_A(h)=h(\T),\qquad r_A(h)=\norm h_\infty.
\]
\end{lemma}
\begin{proof}
Let $\gamma$ be a nonzero character of the algebra $A(\T)$.
It is continuous, has norm one, and satisfies $\gamma(1)=1$.
For $z(x)=e^{ix}$, set $\lambda=\gamma(z)$.
Since $zz^{-1}=1$ and $\norm z_A=\norm{z^{-1}}_A=1$,
we have $|\lambda|\le1$ and $|\lambda|^{-1}\le1$.
Thus $\lambda=e^{ix_0}$ for some $x_0\in\T$.
Absolute convergence of the Fourier series gives
\[
 \gamma(h)=\sum_{k\in\Z}\widehat h(k)e^{ikx_0}=h(x_0).
\]
Conversely, evaluation at each point is a character.
The conclusion follows by identifying the spectrum of an element
of a commutative unital Banach algebra with its values at the
characters.
\end{proof}

\begin{lemma}[Logarithms in the Wiener algebra]\label{lem:logarithm}
If $v\in A(\T)$ and $v=e^{i\psi}$ for a real-valued
$\psi\in C(\T)$, then $\psi\in A(\T)$.
\end{lemma}
\begin{proof}
Choose a real-valued trigonometric polynomial $p$ such that
$\norm{\psi-p}_\infty<1/4$, and set
\[
 w=e^{-ip}v-1\in A(\T).
\]
Then $\norm w_\infty<1/4$. Lemma~\ref{lem:spectrum} gives
$r_A(w)<1/4$. Choose $q$ with $r_A(w)<q<1$.
The spectral radius formula yields $\norm{w^j}_A\le Cq^j$
for all $j\ge1$. Hence the series
\[
 L=\sum_{j=1}^{\infty}\frac{(-1)^{j+1}}{j}\,w^j
\]
converges in the norm of $A(\T)$.
Pointwise, it is the principal branch of $\log(1+w)$.
Since $1+w=e^{i(\psi-p)}$ and $|\psi-p|<1/4<\pi$,
we have $L=i(\psi-p)$. Therefore $\psi=p-iL\in A(\T)$.
\end{proof}

The proof does not identify smallness in the supremum norm with
smallness in the $A$ norm. Convergence of the logarithmic series
in the algebra norm follows from control of the spectral radius.

\begin{lemma}[Removing the degree]\label{lem:degree}
For every $F\in A(\T)$ with $|F|=1$, there exist $m\in\Z$
and a real-valued $\psi\in A(\T)$ such that
\begin{equation}\label{eq:degree-factorization}
 F(x)=e^{imx}e^{i\psi(x)}.
\end{equation}
In particular, for every $n\in\Z$,
\begin{equation}\label{eq:power-norms}
 \norm{F^n}_A=M_\psi(n).
\end{equation}
\end{lemma}
\begin{proof}
The continuous function $F:[0,2\pi]\to\{z:|z|=1\}$ has a
continuous argument $\vartheta$ on the interval. Such an argument
can be constructed by choosing local branches of the argument on
a finite partition and matching their values at the endpoints of
successive subintervals. The equality $F(0)=F(2\pi)$ implies
$\vartheta(2\pi)-\vartheta(0)=2\pi m$ for some $m\in\Z$.
The function $\psi(x)=\vartheta(x)-mx$ has equal endpoint values
and therefore defines a continuous real-valued function on $\T$.
Since $e^{i\psi}=e^{-imx}F\in A(\T)$,
Lemma~\ref{lem:logarithm} gives $\psi\in A(\T)$.
Equality~\eqref{eq:power-norms} follows from the isometry given by
multiplication by the character $e^{inmx}$.
\end{proof}

\subsection{Circle maps, real phases, and real parameters}
\label{subsec:equivalences}

In the usual formulation of the conjecture, a continuous map
$\alpha:\T\to\T$ is represented by $F(x)=e^{i\alpha(x)}$.
Here the exponential is well defined, although $\alpha$ need not
have a real-valued periodic lift. A continuous lift to the real
line satisfies
\[
 \vartheta(x+2\pi)=\vartheta(x)+2\pi m
 \quad\text{for some }m\in\Z.
\]
Thus the conclusion $F(x)=e^{i(mx+b)}$ means precisely that
$\alpha(x)=mx+b$ modulo $2\pi$; see~\cite{LebedevBH,KS}.
For noninteger $t$, the expression $e^{it\alpha}$ is not
intrinsically defined for a circle-valued map. The real-parameter
formulation below therefore uses the periodic phase obtained
\emph{after} removal of the degree.

\begin{lemma}[Comparison of integer and real parameters]
\label{lem:parameter-comparison}
Let $\psi\in A(\T;\R)$ and
$B_\psi=\max_{|u|\le1}M_\psi(u)<\infty$. If $n\in\Z$ and
$|t-n|\le1$, then
\begin{equation}\label{eq:parameter-comparison}
 B_\psi^{-1}M_\psi(n)\le M_\psi(t)\le B_\psi M_\psi(n).
\end{equation}
Consequently,
\begin{equation}\label{eq:integer-real-equivalence}
 \begin{split}
 M_\psi(n)=o(\log|n|)\quad(|n|\to\infty,\ n\in\Z)\\
 \Longleftrightarrow\quad
 M_\psi(t)=o(\log|t|)\quad(|t|\to\infty,\ t\in\R).
 \end{split}
\end{equation}
\end{lemma}
\begin{proof}
The upper bound follows from
$M_\psi(t)\le M_\psi(n)M_\psi(t-n)$, and the lower bound
from $M_\psi(n)\le M_\psi(t)M_\psi(n-t)$.
For each $t$, choose an integer $n$ with $|t-n|\le1/2$.
Then $\log|n|/\log|t|\to1$ as $|t|\to\infty$, so the upper
bound gives the nontrivial implication in
\eqref{eq:integer-real-equivalence}. The reverse implication
is restriction to integer parameters.
\end{proof}

\begin{proposition}[Equivalent rigidity assertions]
\label{prop:equivalent-formulations}
The following assertions are equivalent.
\begin{enumerate}
\item[(i)] Every continuous unimodular $F$ satisfying
$\norm{F^n}_A=o(\log|n|)$ as $|n|\to\infty$, $n\in\Z$,
is a scalar multiple of a character.
\item[(ii)] Every real-valued $\psi\in C(\T)$ satisfying
$\norm{e^{in\psi}}_A=o(\log|n|)$ as $|n|\to\infty$,
$n\in\Z$, is constant.
\item[(iii)] Assertion~\textup{(ii)} holds with $C(\T;\R)$
replaced by $A(\T;\R)$.
\item[(iv)] Every real-valued $\psi\in A(\T)$ satisfying
$\norm{e^{it\psi}}_A=o(\log|t|)$ as $|t|\to\infty$,
$t\in\R$, is constant.
\end{enumerate}
In \textup{(i)} and \textup{(ii)}, the norm is understood to be
$+\infty$ outside $A(\T)$. All four assertions are therefore
consequences of Theorem~\ref{thm:kahane}.
\end{proposition}
\begin{proof}
We establish the equivalences using only the preceding algebraic
facts, without using the proof of Theorem~\ref{thm:kahane}.
If a continuous unimodular $F$ has two consecutive powers in
$A(\T)$, then
\[
 F=F^{n+1}\overline{F^n}\in A(\T).
\]
Thus the growth assumption in~(i) entails $F\in A(\T)$.
Lemma~\ref{lem:degree} gives
$F=e^{imx}e^{i\psi}$ with $\psi\in A(\T;\R)$ and
$\norm{F^n}_A=M_\psi(n)$. Hence (iii) implies (i).
Conversely, apply (i) to $e^{i\psi}$ with
$\psi\in C(\T;\R)$. Since $\psi$ is periodic, this map has
degree zero. Its representation $e^{i\psi}=e^{i(mx+b)}$
therefore has $m=0$, and the continuous function
$\psi-b$ takes values in $2\pi\Z$. It is constant, which
proves (ii). Clearly (ii) implies (iii).

For completeness, the apparently weaker regularity assumption
in (ii) creates no additional case: consecutive large powers
give $e^{i\psi}\in A(\T)$, and
Lemma~\ref{lem:logarithm} then gives $\psi\in A(\T)$.
Finally, (iii) and (iv) are equivalent by
Lemma~\ref{lem:parameter-comparison}.
\end{proof}

The assertion that a non-affine map satisfies
\[
 \limsup_{|n|\to\infty}
 \frac{\norm{F^n}_A}{\log|n|}>0
\]
is the contrapositive of (i), not an additional conjecture.
Similarly, for an invertible $F\in A(\T)$, setting
$\Theta_F(N)=\max_{|n|\le N}\norm{F^n}_A$ gives
\begin{equation}\label{eq:maximum-equivalence}
 \norm{F^n}_A=o(\log|n|)
 \quad\Longleftrightarrow\quad
 \Theta_F(N)=o(\log N).
\end{equation}
Indeed, split the maximum into a fixed finite set and the
indices where $\norm{F^n}_A\le\varepsilon\log|n|$.
These elementary equivalences should be distinguished from the
uniform lower bounds for averages in Theorem~\ref{thm:mean}
and Corollary~\ref{cor:unimodular-blocks}; the latter are stronger
conclusions of the proof, not reductions supplied by this section.

\subsection{Invertible elements, convolution powers, and operators}
\label{subsec:operator-equivalences}

The operator interpretation of quantitative Beurling--Helson
theorems is already explicit in~\cite[Remark~1]{LebedevBH}.
We record the reduction at the logarithmic threshold.

\begin{proposition}\label{prop:operator-equivalences}
Each of the following assertions is equivalent to
Proposition~\ref{prop:equivalent-formulations}\textup{(i)}.
\begin{enumerate}
\item[(a)] If $f$ is invertible in $A(\T)$ and
$\norm{f^n}_A=o(\log|n|)$ as $|n|\to\infty$, $n\in\Z$,
then $f(x)=c e^{imx}$ with $|c|=1$ and $m\in\Z$.
\item[(b)] If $a$ is invertible in the convolution algebra
$\ell^1(\Z)$ and
$\norm{a^{*n}}_1=o(\log|n|)$ as $|n|\to\infty$, $n\in\Z$,
then $a=c\delta_m$ with $|c|=1$ and $m\in\Z$.
\item[(c)] If $U$ is a bounded invertible operator on
$\ell^1(\Z)$ commuting with translations and
$\norm{U^n}_{\ell^1\to\ell^1}=o(\log|n|)$ as $|n|\to\infty$,
$n\in\Z$, then $U=cS_m$, where
$(S_m b)(k)=b(k-m)$ and $|c|=1$.
\end{enumerate}
For negative $n$, the powers in \textup{(a)}--\textup{(c)} are
powers of the respective inverses.
\end{proposition}
\begin{proof}
In (a), the growth assumption and the spectral radius formula give
$r_A(f)\le1$ and $r_A(f^{-1})\le1$.
Lemma~\ref{lem:spectrum} implies $|f(x)|\le1$ and
$|f(x)|^{-1}\le1$ at every point, hence $|f|=1$.
Assertion~(i) applies. Conversely, every unimodular
$f\in A(\T)$ is invertible, with inverse $\overline f$;
and the consecutive-powers argument above reduces a continuous
$F$ in~(i) to this case.

The map
\[
 a\longmapsto f_a(x)=\sum_{k\in\Z}a(k)e^{ikx}
\]
is an isometric isomorphism from the convolution algebra
$\ell^1(\Z)$ onto $A(\T)$, sending $\delta_m$ to $e^{imx}$.
This proves the equivalence of (a) and (b).
If $U$ commutes with translations, put $a=U\delta_0$.
For finitely supported $b$, commutation gives $Ub=a*b$;
density extends this identity to every $b\in\ell^1(\Z)$.
Young's inequality and evaluation at $\delta_0$ yield
\[
 \norm U_{\ell^1\to\ell^1}=\norm a_1.
\]
The inverse of $U$, when it exists, also commutes with
translations and is convolution by the inverse of $a$.
Thus $\norm{U^n}_{\ell^1\to\ell^1}=\norm{a^{*n}}_1$ for every
$n\in\Z$, proving the remaining equivalence.
\end{proof}

The two-sided hypothesis is essential when unimodularity is not
assumed. For example,
$a=\tfrac34\delta_0+\tfrac14\delta_1$ is invertible in
$\ell^1(\Z)$ by a Neumann series and satisfies
$\norm{a^{*n}}_1=1$ for every $n\ge0$, but it is not a point mass.
This does not contradict the one-sided reformulation for
unimodular functions given after Theorem~\ref{thm:kahane}.

\section{The non-harmonic inequality, sampling, and modulation}
\label{sec:tools}

\subsection{Littlewood's inequality on a finite interval}

We use the following precise form of
\cite[Theorem~1.1(iii), p.~1942]{JKS}.

\begin{theorem}[Jaming--Kellay--Saba]\label{thm:JKS}
Let $\lambda_1<\cdots<\lambda_N$ be real numbers satisfying
$\lambda_{r+1}-\lambda_r\ge1$, and let $z_1,\ldots,z_N\in\C$.
For every $L\ge72$,
\begin{equation}\label{eq:JKS}
 \frac1L\int_{-L/2}^{L/2}
 \left|\sum_{r=1}^N z_re^{2\pi i\lambda_r u}\right|\dd u
 \ge\frac1{122}\sum_{r=1}^N\frac{|z_r|}{r+1}.
\end{equation}
\end{theorem}

The constants $72$ and $122$ are independent of $N$ and of the
coefficients. The independence from $N$ in the finite-interval
case is addressed explicitly in \cite[Remark~2.13]{JKS}.
The coefficients may vanish; no assumption
$|z_r|\ge1$ is needed. The fact that the inequality holds on a
finite interval is essential for the argument below. An inequality
only for the limit as $L\to\infty$ with $N$ fixed would not suffice
to choose $N$ as a function of the interval length.

\begin{corollary}\label{cor:JKS-scaled}
If $q_1<\cdots<q_N$ and $q_{r+1}-q_r\ge\delta>0$, then,
for any $z_r\in\C$, $a\in\R$, and $T\ge144\pi/\delta$,
\begin{equation}\label{eq:JKS-scaled}
 \frac1T\int_a^{a+T}
 \left|\sum_{r=1}^N z_re^{itq_r}\right|\dd t
 \ge\frac1{122}\sum_{r=1}^N\frac{|z_r|}{r+1}.
\end{equation}
\end{corollary}
\begin{proof}
In~\eqref{eq:JKS}, take
\[
 \lambda_r=\frac{q_r}{\delta},\qquad
 L=\frac{\delta T}{2\pi},\qquad
 w_r=z_re^{i(a+T/2)q_r}.
\]
The substitution $t=a+T/2+2\pi u/\delta$ gives
$\dd t/T=\dd u/L$, and $|w_r|=|z_r|$.
The condition $L\ge72$ is equivalent to $T\ge144\pi/\delta$.
\end{proof}

\subsection{Exact sampling}

\begin{lemma}\label{lem:sampling}
Let $x_j=2\pi j/N$, $0\le j<N$.
For every $H\in A(\T)$,
\begin{equation}\label{eq:sampling}
 \frac1N\sum_{k=0}^{N-1}
 \left|\sum_{j=0}^{N-1}H(x_j)e^{-2\pi ikj/N}\right|
 \le\norm H_A.
\end{equation}
\end{lemma}
\begin{proof}
Absolute convergence of the Fourier series allows us to interchange
the sums. Orthogonality of the characters of $\Z/N\Z$ gives
\begin{align*}
 \frac1N\sum_{j=0}^{N-1}H(x_j)e^{-2\pi ikj/N}
 &=\sum_{r\in\Z}\widehat H(r)
       \left(\frac1N\sum_{j=0}^{N-1}e^{2\pi i(r-k)j/N}\right)\\
 &=\sum_{\ell\in\Z}\widehat H(k+\ell N).
\end{align*}
Summing the absolute values over $0\le k<N$ and applying the
triangle inequality gives~\eqref{eq:sampling}, since the residue
classes $k+N\Z$ form a disjoint partition of $\Z$.
\end{proof}

This is contractivity of exact restriction to a finite subgroup,
not a Riemann-sum approximation to an integral. In particular,
$N$ need not exceed the oscillation parameter of $H$.

\subsection{A uniform bound for localized modulation}

\begin{lemma}\label{lem:modulation}
Let $\chi\in C_c^\infty((0,2\pi))$, with the products below
extended periodically. Then
\begin{equation}\label{eq:Kchi}
 K_\chi:=\sup_{y\in\R}\norm{\chi(x)e^{-iyx}}_A<\infty.
\end{equation}
For a real-valued $\varphi\in A(\T)$ and $s,t\in\R$,
\begin{equation}\label{eq:modulation-upper}
 \norm{\chi(x)e^{it(\varphi(x)-sx)}}_A
 \le K_\chi M_\varphi(t).
\end{equation}
\end{lemma}
\begin{proof}
Write $y=m+r$, where $m\in\Z$ and $|r|\le1/2$.
Invariance under integer modulation gives
\[
 \norm{\chi e^{-iyx}}_A=\norm{\chi e^{-irx}}_A.
\]
Since $\chi$ vanishes near the endpoints, the function
$\chi e^{-irx}$ has a smooth periodic extension.
By~\eqref{eq:sobolev}, uniformly in $r$,
\[
 \norm{\chi e^{-irx}}_A
 \le C_S\left(\tfrac32\norm\chi_2+\norm{\chi'}_2\right).
\]
This proves~\eqref{eq:Kchi}.
Inequality~\eqref{eq:modulation-upper} follows by
applying~\eqref{eq:algebra} to the product
$e^{it\varphi}(\chi e^{-itsx})$.
\end{proof}

We do not assume that $e^{-iyx}$ is periodic for every $y\in\R$.
It is the \emph{localized} product that is extended periodically.
The cutoff is precisely what makes the estimate independent of
the possibly arbitrarily large product $ts$.

\section{Phase localization and separation of samples}\label{sec:geometry}

In this section, we assume that the continuous real-valued phase
$\varphi$ is normalized by
\begin{equation}\label{eq:normalization}
 \min\varphi=0,\qquad \max\varphi=1,\qquad
 \varphi(0)=\varphi(2\pi)=1.
\end{equation}
We shall remove this normalization in the proof of
Theorem~\ref{thm:mean}.

By continuity, choose $0<\varepsilon<\pi/4$ so that
$\varphi\ge3/4$ on
$[0,2\varepsilon]\cup[2\pi-2\varepsilon,2\pi]$.
Fix $\chi\in C_c^\infty((0,2\pi))$ such that
$0\le\chi\le1$ and $\chi=1$ on
$[2\varepsilon,2\pi-2\varepsilon]$.
Then
\begin{equation}\label{eq:cutoff-level}
 \chi(x)<1\quad\Longrightarrow\quad\varphi(x)\ge\tfrac34.
\end{equation}
Since the minimum of the phase is zero, there is a nonempty open
interval $J\subset(0,2\pi)$ on which $\varphi<1/4$.
In particular, $\chi=1$ on $J$.
Set
\begin{equation}\label{eq:S}
 S=\left[-\frac1{16\pi},\frac1{16\pi}\right],
 \qquad q_s(x)=\varphi(x)-sx\quad(s\in S).
\end{equation}
For $x\in[0,2\pi]$, we have $|sx|\le1/8$, and hence
\begin{align}
 x\in J&\quad\Longrightarrow\quad q_s(x)<\tfrac38,
       \label{eq:low}\\
 \chi(x)<1&\quad\Longrightarrow\quad q_s(x)\ge\tfrac58.
       \label{eq:high}
\end{align}

\begin{lemma}[Initial coefficients]\label{lem:first-coefficients}
Let $x_j=2\pi j/N$, and suppose that, for some $s\in S$,
the values $q_s(x_j)$ are pairwise distinct.
Arrange them as
$q_s(x_{j_1})<\cdots<q_s(x_{j_N})$, and write
$m_N=\#\{j:x_j\in J\}$.
Then
\begin{equation}\label{eq:first-coefficients}
 \chi(x_{j_r})=1\qquad(1\le r\le m_N).
\end{equation}
Moreover, for all $N\ge8\pi/|J|$,
\begin{equation}\label{eq:rho}
 m_N\ge\rho N,\qquad \rho=\frac{|J|}{4\pi}>0.
\end{equation}
\end{lemma}
\begin{proof}
By~\eqref{eq:low}, at least $m_N$ of the values $q_s(x_j)$
are less than $3/8$. Thus the first $m_N$ values in the full
ordering are also less than $3/8$. By~\eqref{eq:high}, the
corresponding values of $\chi$ equal one.
There is no need for the first $m_N$ indices to correspond
exactly to points in $J$.

The number of grid points in the open interval $J$ is at least
$N|J|/(2\pi)-2$. For $N\ge8\pi/|J|$, this
gives~\eqref{eq:rho}.
\end{proof}

The lemma controls the beginning of the ordering, not just the
number of coefficients of modulus one. This distinction is
essential for the sum $\sum_r|z_r|/(r+1)$: a positive proportion
of large coefficients placed only at the final indices would
not yield a logarithmic factor.

\begin{lemma}[Separation by a single parameter]\label{lem:separation}
For every $N\ge2$, there exists $s_N\in S$ such that
\begin{equation}\label{eq:separation}
 |\varphi(x_j)-s_Nx_j-\varphi(x_\ell)+s_Nx_\ell|
 \ge\delta_N:=\frac1{16N^3}\qquad(j\ne\ell).
\end{equation}
\end{lemma}
\begin{proof}
For $j<\ell$, the parameters for which the inequality with
lower bound $\delta>0$ fails form the open interval
\[
 |\varphi(x_j)-\varphi(x_\ell)-s(x_j-x_\ell)|<\delta.
\]
Its length is
$2\delta/|x_j-x_\ell|=\delta N/(\pi(\ell-j))$.
The sum of the lengths of their intersections with $S$ is
therefore at most
\begin{equation}\label{eq:bad-parameters}
 \frac{\delta N}{\pi}
 \sum_{d=1}^{N-1}\frac{N-d}{d}
 \le\frac{\delta N^3}{\pi}.
\end{equation}
For $\delta=1/(16N^3)$, this bound is $1/(16\pi)$,
whereas $|S|=1/(8\pi)$.
Thus there is a parameter $s_N\in S$ outside the union of all
the bad intervals, and it satisfies~\eqref{eq:separation}.
\end{proof}

No separation of the values $\varphi(x_j)$ is assumed; these
values may repeat. The difference $x_j-x_\ell$ is the ordinary
difference of the representatives in $[0,2\pi)$, not the distance
on the circle. For each $N$, we choose a single parameter $s_N$
and keep it fixed throughout the subsequent integration.

\section{Proof of the logarithmic lower bound for averages}\label{sec:meanproof}

\begin{proof}[Proof of Theorem~\ref{thm:mean}]
We first impose the normalization~\eqref{eq:normalization}.
Fix $\chi$, $J$, and $S$ as in Section~\ref{sec:geometry}.
In particular, $K_\chi$, $\rho$, and the threshold
in~\eqref{eq:rho} are independent of $N$, $T$, and the
shift $a$.

For sufficiently large $N$, choose $s_N$ as in
Lemma~\ref{lem:separation}, put
$q_j=\varphi(x_j)-s_Nx_j$, and define, for $0\le k<N$,
\begin{equation}\label{eq:P}
 P_{k,N}(t)=\sum_{j=0}^{N-1}
       \chi(x_j)e^{-2\pi ikj/N}e^{itq_j}.
\end{equation}
Applying Lemmas~\ref{lem:sampling} and~\ref{lem:modulation} to
\[
 H_{N,t}(x)=\chi(x)e^{it(\varphi(x)-s_Nx)},
\]
we obtain, for every $t\in\R$,
\begin{equation}\label{eq:P-upper}
 \frac1N\sum_{k=0}^{N-1}|P_{k,N}(t)|
 \le\norm{H_{N,t}}_A\le K_\chi M_\varphi(t).
\end{equation}
The inclusion $H_{N,t}\in A(\T)$ follows from the factorization
$H_{N,t}=e^{it\varphi}(\chi e^{-its_Nx})$;
it does not require the function $x\mapsto s_Nx$ itself to
be periodic.

Now arrange the frequencies in increasing order:
$q_{j_1}<\cdots<q_{j_N}$.
Their consecutive gaps are at least $\delta_N=1/(16N^3)$.
The coefficient of the frequency $q_{j_r}$ in~\eqref{eq:P} is
\[
 z_{r,k}=\chi(x_{j_r})e^{-2\pi ikj_r/N},
 \qquad |z_{r,k}|=\chi(x_{j_r}).
\]
Its modulus is independent of $k$.
Lemma~\ref{lem:first-coefficients} gives $|z_{r,k}|=1$
for $1\le r\le m_N$, where $m_N\ge\rho N$.
We may therefore apply Corollary~\ref{cor:JKS-scaled} whenever
\begin{equation}\label{eq:T-threshold}
 T\ge\frac{144\pi}{\delta_N}=2304\pi N^3.
\end{equation}
For every $a\in\R$ and $0\le k<N$, this gives
\begin{align}
 \frac1T\int_a^{a+T}|P_{k,N}(t)|\dd t
 &\ge\frac1{122}\sum_{r=1}^{m_N}\frac1{r+1}\notag\\
 &\ge\frac1{122}\bigl(\log N-C_J\bigr),
       \label{eq:P-lower}
\end{align}
where we may take $C_J=1-\log\rho$.
Indeed, $\sum_{r=1}^{m}1/(r+1)\ge\log(m+2)-1$.

Integrating~\eqref{eq:P-upper} and then
using~\eqref{eq:P-lower} for each $k$, we obtain
\begin{equation}\label{eq:mean-N}
 \mathcal Q_\varphi(a,T)
 \ge\frac1{122K_\chi}(\log N-C_J)
 \qquad(T\ge2304\pi N^3).
\end{equation}
At this step, we sum $N$ lower bounds and divide by $N$;
there is no additional loss depending on the size of the grid.

For any sufficiently large $T$, choose
\[
 N=\left\lfloor\left(\frac{T}{2304\pi}\right)^{1/3}\right\rfloor.
\]
Then~\eqref{eq:T-threshold} holds, and, for large $T$,
\[
 \log N\ge\tfrac13\log T-\tfrac13\log(2304\pi)-\log2.
\]
Consequently,~\eqref{eq:mean-N} yields
\begin{equation}\label{eq:mean-additive}
 \mathcal Q_\varphi(a,T)
 \ge\frac1{366K_\chi}\log T-C_\varphi^{(0)},
\end{equation}
with a constant $C_\varphi^{(0)}$ independent of $a$ and $T$.
Increasing the threshold $T_\varphi$
gives~\eqref{eq:mean-statement}; in the normalized case, we may
take $c_\varphi=1/(732K_\chi)$.

It remains to remove the normalization.
For the original nonconstant $\varphi\in A(\T;\R)$, let
$d=\max\varphi-\min\varphi>0$, and let $x_*$ be a point
where the maximum is attained. The function
\[
 \Phi(x)=\frac{\varphi(x+x_*)-\min\varphi}{d}
\]
belongs to $A(\T)$ and satisfies~\eqref{eq:normalization}.
Translation invariance and invariance under multiplication by
unimodular constants give $M_\varphi(t)=M_\Phi(dt)$, hence
\[
 \mathcal Q_\varphi(a,T)=\mathcal Q_\Phi(da,dT).
\]
Applying~\eqref{eq:mean-additive} to $\Phi$ and using
$\log(dT)=\log T+\log d$, we again obtain a lower bound equal
to a positive multiple of $\log T$ minus a constant.
Absorbing this constant by increasing the threshold completes
the proof. Uniformity in $a$ is preserved.
\end{proof}

\begin{remark}
The constants in the proof depend on the fixed phase through the
choice of the cutoff $\chi$, the length of the interval $J$,
and the initial rescaling. The parameter $s_N$ depends on the
grid, but estimate~\eqref{eq:modulation-upper} is uniform over
all $s,t\in\R$. This uniformity allows us to compare the
sums~\eqref{eq:P} with a single function $M_\varphi$, even
though the auxiliary phase depends on $N$.
\end{remark}

\section{Integer powers and the proof of Kahane's conjecture}
\label{sec:kahaneproof}

\begin{corollary}\label{cor:integer-blocks}
If $\varphi\in A(\T;\R)$ is nonconstant, then there exist
$c'_\varphi>0$ and $N_\varphi\in\N$ such that
\begin{equation}\label{eq:integer-blocks}
 \frac1N\sum_{n=j}^{j+N-1}M_\varphi(n)
 \ge c'_\varphi\log N
 \qquad(j\in\Z,\ N\ge N_\varphi).
\end{equation}
\end{corollary}
\begin{proof}
By continuity of $M_\varphi$, the number
$B_\varphi=\max_{0\le u\le1}M_\varphi(u)$ is finite.
For $n\in\Z$ and $0\le u\le1$, submultiplicativity gives
$M_\varphi(n+u)\le B_\varphi M_\varphi(n)$.
Summing the integrals over consecutive unit intervals, we obtain
\[
 \frac1N\int_j^{j+N}M_\varphi(t)\dd t
 \le\frac{B_\varphi}{N}\sum_{n=j}^{j+N-1}M_\varphi(n).
\]
Theorem~\ref{thm:mean} proves the assertion with
$c'_\varphi=c_\varphi/B_\varphi$, after increasing the integer
threshold if necessary.
\end{proof}

\begin{corollary}\label{cor:unimodular-blocks}
Let $F\in A(\T)$ satisfy $|F|=1$, and suppose that $F$ is not
of the form $e^{i(mx+b)}$.
There exist $c_F>0$ and $N_F\in\N$ such that
\begin{equation}\label{eq:unimodular-blocks}
 \frac1N\sum_{n=j}^{j+N-1}\norm{F^n}_A
 \ge c_F\log N
 \qquad(j\in\Z,\ N\ge N_F).
\end{equation}
\end{corollary}
\begin{proof}
By Lemma~\ref{lem:degree}, we have $F=e^{imx}e^{i\psi}$
with a real-valued $\psi\in A(\T)$.
The function $\psi$ is nonconstant, since otherwise $F$ would
have the excluded form.
Now~\eqref{eq:power-norms} and
Corollary~\ref{cor:integer-blocks} give~\eqref{eq:unimodular-blocks}.
\end{proof}

\begin{proof}[Proof of Theorem~\ref{thm:kahane}]
Condition~\eqref{eq:kahane-condition}, restricted to positive $n$,
implies that $F^n\in A(\T)$ for all sufficiently large $n$.
Choosing two consecutive such powers and using $|F|=1$, we obtain
\[
 F=F^{n+1}\overline{F^n}\in A(\T).
\]
Suppose that $F$ is not of the form $e^{i(mx+b)}$.
By Corollary~\ref{cor:unimodular-blocks}, for large $N$,
\begin{equation}\label{eq:contradiction-lower}
 \frac1N\sum_{n=0}^{N-1}\norm{F^n}_A\ge c_F\log N.
\end{equation}
On the other hand, for every $\eta>0$ there exists $n_\eta\ge2$
such that $\norm{F^n}_A\le\eta\log n$ for $n\ge n_\eta$.
Therefore
\[
 \frac1N\sum_{n=0}^{N-1}\norm{F^n}_A
 \le\frac1N\sum_{n=0}^{n_\eta-1}\norm{F^n}_A+\eta\log N
 \qquad(N>n_\eta).
\]
Dividing by $\log N$, letting $N$ tend to infinity, and then
letting $\eta$ tend to zero, we obtain
\[
 \frac1N\sum_{n=0}^{N-1}\norm{F^n}_A=o(\log N),
\]
contrary to~\eqref{eq:contradiction-lower}.
Thus $F(x)=e^{i(mx+b)}$.
\end{proof}

\section{Consequences for homomorphisms and higher-dimensional tori}
\label{sec:consequences}

\subsection{Composition operators and L\'evy's endomorphism problem}

The classification of endomorphisms of $A(\T)$ was a problem of
P.~L\'evy, solved by the Beurling--Helson theorem; see
\cite[Introduction]{KS}. Thus its recovery below is a classical
consequence, not a separate previously open problem.
The quantitative extension concerns the growth allowed for the
norms of restrictions to finite-dimensional spaces of trigonometric
polynomials.

Let $\alpha:\T\to\T$ be continuous, put $F=e^{i\alpha}$,
and suppose $F\in A(\T)$. For $N\ge1$, let
\[
 \mathcal P_N=\operatorname{span}\{e^{inx}:|n|\le N\},
 \qquad
 C_{\alpha,N}:\mathcal P_N\to A(\T),\quad
 C_{\alpha,N}p=p\circ\alpha,
\]
where $\mathcal P_N$ carries the $A(\T)$ norm.
Then
\begin{equation}\label{eq:composition-norm}
 \norm{C_{\alpha,N}}=\max_{|n|\le N}\norm{F^n}_A
 =\Theta_F(N).
\end{equation}
The upper bound is the triangle inequality applied to the Fourier
series of $p$, and the reverse bound follows by taking a single
character. In particular, \eqref{eq:maximum-equivalence} shows
that the rigidity assertion with
$\norm{C_{\alpha,N}}=o(\log N)$ is another equivalent formulation
of Kahane's conjecture.

\begin{corollary}\label{cor:composition-growth}
Let $\alpha:\T\to\T$ be continuous with $e^{i\alpha}\in A(\T)$.
If $\alpha$ is not affine, then there are $c_\alpha>0$ and
$N_\alpha$ such that
\[
 \norm{C_{\alpha,N}}\ge c_\alpha\log N
 \qquad(N\ge N_\alpha).
\]
Consequently, the only continuous maps inducing bounded
composition operators $A(\T)\to A(\T)$ are
$\alpha(x)=mx+b$ modulo $2\pi$.
\end{corollary}
\begin{proof}
By Corollary~\ref{cor:unimodular-blocks} and
\eqref{eq:composition-norm},
\[
 \norm{C_{\alpha,N}}
 \ge\frac1N\sum_{n=0}^{N-1}\norm{F^n}_A
 \ge c_F\log N.
\]
A bounded composition operator has uniformly bounded restrictions
$C_{\alpha,N}$ and sends $e^{ix}$ to $F\in A(\T)$, so $\alpha$
must be affine. Conversely, an affine map induces a contractive
composition operator, as is immediate from the Fourier series.
\end{proof}

\subsection{Homomorphisms from weighted Wiener algebras}

Let $w:\Z\to[1,\infty)$ satisfy $w(0)=1$ and
$w(k+\ell)\le w(k)w(\ell)$. Define the weighted Wiener algebra
\[
 A_w(\T)=\left\{f\in C(\T):\norm f_{A_w}
 :=\sum_{n\in\Z}|\widehat f(n)|w(n)<\infty\right\}.
\]
It is a unital Banach algebra, continuously embedded in $A(\T)$,
and trigonometric polynomials are dense in it.

\begin{corollary}[Sublogarithmic weights]\label{cor:weighted-homomorphisms}
Suppose that $w(n)=o(\log|n|)$ as $|n|\to\infty$.
Every nonzero bounded complex-linear algebra homomorphism
$T:A_w(\T)\to A(\T)$ has the form
\begin{equation}\label{eq:weighted-homomorphism}
 (Tf)(x)=f(mx+b),\qquad m\in\Z,\quad b\in\R.
\end{equation}
In particular, $T$ extends to a contractive endomorphism of
$A(\T)$, and $\norm T_{A_w\to A}=1$.
\end{corollary}
\begin{proof}
The continuous function $T1$ is idempotent and hence is constant,
since $\T$ is connected. As $T$ is nonzero, $T1=1$.
Put $z(x)=e^{ix}$ and $F=Tz$. Then $F$ is invertible with inverse
$T(z^{-1})$, and for every $n\in\Z$,
\[
 \norm{F^n}_A=\norm{T(z^n)}_A\le\norm T\,w(n).
\]
Proposition~\ref{prop:operator-equivalences}(a), now established
by Theorem~\ref{thm:kahane}, gives $F(x)=e^{i(mx+b)}$.
Multiplicativity proves~\eqref{eq:weighted-homomorphism} for
trigonometric polynomials. Density and the inequality
\[
 \norm{f(mx+b)}_A\le\norm f_A\le\norm f_{A_w}
\]
extend it to all of $A_w(\T)$.
The same formula defines the claimed extension to $A(\T)$.
Finally, contractivity and $T1=1$ give $\norm T=1$.
\end{proof}

For example, the weights
\[
 w_\beta(n)=\bigl(1+\log(1+|n|)\bigr)^\beta,
 \qquad 0\le\beta\le1,
\]
are submultiplicative. Indeed,
$1+|k+\ell|\le(1+|k|)(1+|\ell|)$ and
$1+a+b\le(1+a)(1+b)$ for $a,b\ge0$.
Corollary~\ref{cor:weighted-homomorphisms} applies when
$0\le\beta<1$. The logarithmic boundary cannot be included:
for $\alpha(x)=|x|$ on $[-\pi,\pi]$, interpreted modulo $2\pi$,
Example~\ref{ex:sharpness} gives
\[
 \norm{e^{in\alpha}}_A\le Cw_1(n)\qquad(n\in\Z).
\]
Consequently the series
$\sum_n\widehat f(n)e^{in\alpha}$ converges in $A(\T)$ for
every $f\in A_{w_1}(\T)$ and defines the non-affine bounded
homomorphism $f\mapsto f\circ\alpha$ into $A(\T)$.
Thus the distinction between sublogarithmic and logarithmic
weights is sharp for this class of homomorphisms.

\subsection{Higher-dimensional tori}

The restriction argument for passing from the circle to a
finite-dimensional torus is standard; it is given for Lebedev's
quantitative theorem in~\cite[Remark~2]{LebedevBH}.
We include the details because the same argument also transfers
our stronger bounds for averages.

For $d\ge1$, equip $A(\T^d)$ with the norm
$\norm H_{A(\T^d)}=\sum_{k\in\Z^d}|\widehat H(k)|$.
As in the one-dimensional case, we set $\norm H_{A(\T^d)}=+\infty$
for $H\in C(\T^d)\setminus A(\T^d)$.
If all but one coordinate are fixed, restriction is contractive:
\begin{equation}\label{eq:coordinate-restriction}
 \norm{H(y_1,\ldots,y_{j-1},\,\cdot\,,y_{j+1},\ldots,y_d)}_{A(\T)}
 \le\norm H_{A(\T^d)}.
\end{equation}
To see this, sum the absolutely convergent Fourier series over
the fixed coordinates first and apply the triangle inequality
to each resulting Fourier coefficient.

\begin{corollary}\label{cor:higher-torus}
Let $F\in C(\T^d)$ be unimodular. If
\[
 \norm{F^n}_{A(\T^d)}=o(\log|n|)
 \qquad(|n|\to\infty,\ n\in\Z),
\]
then
\begin{equation}\label{eq:torus-character}
 F(x)=e^{i(\langle m,x\rangle+b)}
 \quad\text{for some }m\in\Z^d,\ b\in\R.
\end{equation}
More strongly, if $F\in A(\T^d)$ is unimodular and is not of
this form, then there exist $c_F>0$ and $N_F$ such that
\begin{equation}\label{eq:torus-blocks}
 \frac1N\sum_{n=j}^{j+N-1}\norm{F^n}_{A(\T^d)}
 \ge c_F\log N\qquad(j\in\Z,\ N\ge N_F).
\end{equation}
\end{corollary}
\begin{proof}
Under the first hypothesis, two consecutive large powers show
that $F\in A(\T^d)$. Every coordinate-circle restriction then
satisfies the hypothesis of Theorem~\ref{thm:kahane}, by
\eqref{eq:coordinate-restriction}.
It remains to check that if every such restriction is a scalar
multiple of a character, then $F$ has the global
form~\eqref{eq:torus-character}.

Fix a coordinate index $j$ and denote the remaining coordinates by $y$.
Write its restriction as $F_y(t)=c(y)e^{im_j(y)t}$.
Its degree $m_j(y)$ is locally constant in $y$: uniformly close
unimodular functions have the same degree, because their quotient
has a continuous argument with values in $(-\pi,\pi)$.
Uniform continuity of $F$ and connectedness of $\T^{d-1}$ show
that $m_j(y)$ is independent of $y$; denote it by $m_j$.
It follows that $F(x)e^{-i\langle m,x\rangle}$ is independent
of every coordinate and is therefore constant. This proves the
first assertion.

The same observation shows that any $F$ not of the
form~\eqref{eq:torus-character} has a coordinate-circle restriction
$G$ that is not a scalar multiple of a character. Apply
Corollary~\ref{cor:unimodular-blocks} to $G$ and use
$\norm{G^n}_{A(\T)}\le\norm{F^n}_{A(\T^d)}$ term by term.
\end{proof}

Similarly, for a nonconstant $\psi\in A(\T^d;\R)$, at least
one coordinate-circle restriction is nonconstant. Applying
Theorem~\ref{thm:mean} to that restriction proves
\[
 \frac1T\int_a^{a+T}\norm{e^{it\psi}}_{A(\T^d)}\dd t
 \ge c_\psi\log T\qquad(a\in\R,\ T\ge T_\psi).
\]
Applying Corollary~\ref{cor:higher-torus} componentwise also
shows that a continuous map $\alpha:\T^d\to\T^r$ whose
coordinate functions satisfy
$\norm{e^{in\alpha_\ell}}_{A(\T^d)}=o(\log|n|)$ for
$1\le\ell\le r$ is affine:
$\alpha(x)=Mx+b$ modulo $2\pi$, with an integer $r\times d$
matrix $M$. Via convolution, the analogue of
Proposition~\ref{prop:operator-equivalences}\textup{(c)}, with the same
hypotheses, holds on $\ell^1(\Z^d)$, with conclusion $U=cS_m$,
$|c|=1$, and $m\in\Z^d$.

\subsection{Rigidity from a sequence of long averaging intervals}

The uniformity in the position of the interval yields a
strengthening not contained in the elementary equivalences of
Section~\ref{sec:algebra}.
If $F\in A(\T)$ is unimodular and there are
$j_k\in\Z$ and integers $N_k\to\infty$ such that
\begin{equation}\label{eq:small-blocks}
 \frac1{N_k\log N_k}
 \sum_{n=j_k}^{j_k+N_k-1}\norm{F^n}_A\longrightarrow0,
\end{equation}
then $F$ is a scalar multiple of a character.
Indeed, otherwise Corollary~\ref{cor:unimodular-blocks} bounds
the left-hand side below by $c_F>0$ for all sufficiently large
$k$. No restriction on the positions $j_k$ is needed.
The analogous continuous-parameter condition
\[
 \frac1{T_k\log T_k}\int_{a_k}^{a_k+T_k}M_\psi(t)\dd t
 \longrightarrow0,\qquad T_k\to\infty,
\]
forces a real $\psi\in A(\T)$ to be constant, by
Theorem~\ref{thm:mean}. These conclusions concern long blocks,
not merely an arbitrary sparse sequence of individual powers.

\section{Sharpness of the threshold and scope of the conclusions}\label{sec:remarks}

\subsection{Why big \texorpdfstring{$O$}{O} is not enough}

Kahane's classical examples~\cite[Ch.~VI, \S2, pp.~75--76]{Kahane} show that
$o(\log|n|)$ in Theorem~\ref{thm:kahane} cannot be replaced by
$O(\log|n|)$. For completeness, we reproduce the calculation for Kahane's
tent phase $x\mapsto|x|$.

\begin{example}\label{ex:sharpness}
Identify the circle with $[-\pi,\pi]$ with its endpoints
identified, and set $F(x)=e^{i|x|}$.
Then $F\in A(\T)$, $|F|=1$, $F$ is not a scalar multiple of
a character, and
\begin{equation}\label{eq:sharpness}
 \norm{F^n}_A\asymp\log(|n|+2)\qquad(n\in\Z).
\end{equation}
\end{example}
\begin{proof}
The function $F$ is periodic and belongs to $W^{1,2}(\T)$,
and hence to $A(\T)$ by~\eqref{eq:sobolev}.
For $n\ge1$, direct integration gives
\begin{equation}\label{eq:triangle-coefficients}
 \widehat{F^n}(k)=
 \begin{cases}
 \dfrac12,& k=\pm n,\\[5pt]
 \dfrac{n\bigl((-1)^{n+k}-1\bigr)}{\pi i(n^2-k^2)},
       & k\ne\pm n.
 \end{cases}
\end{equation}
Thus the nonzero coefficients other than those at $k=\pm n$
have modulus $2n/(\pi|n^2-k^2|)$.
For $0\le k\le2n$, $k\ne n$, this is at most
$2/(\pi|n-k|)$. Summation and the symmetry
$k\leftrightarrow-k$ give a bound of $O(\log(n+2))$.
The sum over $|k|>2n$ is $O(1)$, since the coefficients are
bounded by $Cn/k^2$.

For the lower bound, it suffices to take $k=n-d$, where
$1\le d\le n$ is odd. The corresponding coefficient has modulus
\[
 \frac{2n}{\pi d(2n-d)}\ge\frac1{\pi d}.
\]
The sum over such $d$ has order $\log(n+2)$ for large $n$.
Adjusting the constant for the finitely many remaining positive
integers gives~\eqref{eq:sharpness} for $n\ge1$.
For $n<0$, use $\norm{F^{-n}}_A=\norm{F^n}_A$;
for $n=0$, use $\norm1_A=1$.
Finally, the argument of $F$ has different slopes on
$(-\pi,0)$ and $(0,\pi)$, so $F$ is not of the form $e^{i(mx+b)}$.
\end{proof}

\subsection{Averages versus pointwise estimates}

For a nonconstant real-valued $\varphi\in A(\T)$,
Theorem~\ref{thm:mean} with $a=T$ gives
\[
 \max_{T\le t\le2T}M_\varphi(t)\ge c_\varphi\log T
 \qquad(T\ge T_\varphi).
\]
Similarly,~\eqref{eq:unimodular-blocks} implies that, for a
unimodular $F\in A(\T)$ that is not a scalar multiple of a
character,
\[
 \limsup_{|n|\to\infty}\frac{\norm{F^n}_A}{\log|n|}>0.
\]
This suffices to settle Kahane's conjecture, whose assumption
requires slow growth of the powers $F^n$ for \emph{all}
sufficiently large $|n|$.
A lower bound for averages should not, however, be identified
with the inequality $M_\varphi(t)\ge c_\varphi\log t$ for every
large $t$. The argument above does not determine whether there
are sparse sequences of parameters along which the norms remain
small, nor does it settle the question of whether
$M_\varphi(t)\to\infty$ for every nonconstant
$\varphi\in A(\T;\R)$.

This last distinction has an explicit precedent in the literature.
Lebedev~\cite[Remark~5]{LebedevBH} asks about bounded norms along
an unbounded sequence of integer powers of a nonlinear continuous
circle map, and proves that such a sequence is impossible when
the map is absolutely continuous. Thus qualitative divergence in
the absolutely continuous, and in particular $C^1$, case is
already known. A logarithmic lower bound for \emph{each}
sufficiently large parameter is a different assertion, and is
not proved here even for arbitrary nonconstant $C^1$ phases.

\subsection{Related problems that require additional arguments}

The homomorphism and torus consequences above are direct
consequences of the theorem and its averaging estimate.
They should not be confused with solutions of other extremal
problems that occur in the same literature.

First, the Littlewood problem modulo a prime, studied in
\cite{GK,KS}, concerns arbitrary subsets of a finite cyclic
group and lower bounds for the norms of their indicator
functions. Our sampling argument uses finite groups, but its
conclusion controls powers of a \emph{fixed continuous phase},
with constants depending on that phase. It supplies no uniform
lower bound for arbitrary indicators on groups of increasing
prime order. Moreover, the non-harmonic Littlewood inequality
\cite{JKS} is an input to the proof, not a new consequence of it.
In particular, the proof does not resolve the sharper
Dirichlet-kernel extremal question stated as Question~A
in~\cite{JKS}.

Second, one cannot transfer the rigidity statement to a weaker
norm simply from a norm comparison. For example, let $U(\T)$
denote the space of continuous functions with uniformly convergent
Fourier series, with
\[
 \norm f_{U(\T)}=\sup_{N\ge0}\norm{S_N f}_\infty,
 \qquad S_N f(x)=\sum_{|k|\le N}\widehat f(k)e^{ikx}.
\]
Although $\norm f_{U(\T)}\le\norm f_A$ for $f\in A(\T)$,
a lower bound in $A(\T)$ gives no lower bound in $U(\T)$.
Indeed, Lebedev~\cite{LebedevUniform} recalls the existence of
non-affine circle maps with bounded $U(\T)$ norms of their
powers, while proving logarithmic growth in that norm for
non-affine piecewise linear maps. This provides a concrete
boundary to the scope of the present argument.

Finally, the bound $\norm{F^n}_A=O(\log|n|)$ is compatible with
our lower bounds and so is not a rigidity hypothesis of the kind
proved here. Example~\ref{ex:sharpness} establishes sharpness of
the little-$o$ threshold, but neither it nor the averaging proof
classifies all maps with logarithmically bounded powers.

\section*{Use of artificial intelligence}
This work was supported by OpenAI through the ChatGPT for Academic Researchers program.
AI-assisted tools were used during the preparation of this manuscript for mathematical verification and editorial support. The author has independently verified the content and takes full responsibility for its accuracy and conclusions.

\begingroup
\small
\setlength{\bibsep}{5pt}

\endgroup
\end{document}